\documentclass[10pt,a4paper,reqno]{amsart}
\usepackage[T1]{fontenc}
\usepackage{amssymb}
\usepackage{amsmath}
\usepackage{amsfonts}
\usepackage{ifthen}
\usepackage{leftidx}
\usepackage{comment}
\usepackage{mathtools}
\usepackage{enumitem}
\newlist{myitemize}{itemize}{1}
\setlist[myitemize]{leftmargin=0.24in}

\newcounter{margnotes}

\def\sideremark#1{\ifvmode\leavevmode\fi\vadjust{\vbox to0pt{\vss 
			\hbox to 0pt{\hskip\hsize\hskip1em           
				\vbox{\hsize3cm\tiny\raggedright\pretolerance10000
					\noindent #1\hfill}\hss}\vbox to8pt{\vfil}\vss}}}%

\newcounter{lemenumi}
\newcommand{\labelemenumi}{(\alph{lemenumi})}

\newtheorem{theorem}{Theorem}[section]
\newtheorem{lemma}[theorem]{Lemma}

\newtheorem{proposition}[theorem]{Proposition}
\newtheorem{corollary}[theorem]{Corollary}

\theoremstyle{definition}
\newtheorem{definition}[theorem]{Definition}

\theoremstyle{remark}
\newtheorem{remark}[theorem]{Remark}

\numberwithin{equation}{section}

\newcommand{\Mon}{\mathrm{Mon}}
\newcommand{\Var}{\mathrm{Var}}

\newcommand{\OO}{\mathcal{O}}

\begin{document}
	\newcommand{\xc}{\xi}
	\newcommand{\yc}{\eta}
	\newcommand{\ep}{\varepsilon}
	\newcommand{\jb}{II}
	\newcommand{\df}{d}
	\newcommand{\al}{\alpha}
	
	\newcommand{\testrat}{R}
	\newcommand{\testpol}{P}
	
	\newcommand{\pn}{Q}
	\newcommand{\qn}{Q_1}
	\newcommand{\mon}{\mathcal{M}on}
	\newcommand{\ve}{\varepsilon}
	\newcommand{\var}{\mathcal{V}ar}
	\newcommand{\org}{T}
	\newcommand{\orgg}{t}
	\newcommand{\supp}{\text{supp}}
	\newcommand{\R}{\mathbb{R}}
	\newcommand{\C}{\mathbb{C}}
	\newcommand{\I}{\mathbb{I}}
	\newcommand{\N}{\mathbb{N}}
	\newcommand{\Z}{\mathbb{Z}}
	\newcommand{\Q}{\mathbb{Q}}
	\newcommand{\Xbar}{\bold{X}}
	
	
	\newcommand{\edg}{\gamma^0}
	\newcommand{\vtx}{p}
	\newcommand{\CH}{{\mathbb{C}} H_1({\mathcal{O}})}
	\newcommand{\CHy}{{\mathbb{C}} H_1({\mathcal{O}_y})}
	
		\title[Infinite Orbit depth and length of Melnikov functions]{Infinite 
		Orbit depth \\and\\ length of Melnikov functions}
	
\author[P. Marde\v si\'c]{Pavao Marde\v si\'c}
\address{Universit\'e de Bourgogne, Institute de 
	Math\'ematiques de Bourgogne - UMR 5584 CNRS\\
	Universit\'e de Bourgogne-Franche-Comt\'e,
	9 avenue Alain Savary,
	BP 47870, 21078 Dijon\\France}
\email{mardesic@u-bourgogne.fr}

\author[D. Novikov]{Dmitry Novikov}
\address{Faculty of Mathematics and 
	Computer Science, Weizmann Institute of Science, Rehovot, 7610001 Israel}
\email{dmitry.novikov@weizmann.ac.il}

\author[L. Ortiz-Bobadilla]{Laura Ortiz-Bobadilla} 
\address{Instituto de Matem\'aticas, Universidad Nacional Aut\'onoma de 
	M\'exico 	(UNAM), 	\'Area de la Investigaci\'on Cient\'ifica, Circuito 
	exterior, Ciudad 	Universitaria, 04510, Ciudad de M\'exico, M\'exico}
\email{laura@matem.unam.mx}

\author[J. Pontigo-Herrera]{Jessie Pontigo-Herrera}
\address{Faculty of Mathematics and 
	Computer Science, Weizmann Institute of Science, Rehovot, 7610001 Israel}
\email{jessie-diana.pontigo-herrera@weizmann.ac.il}

\thanks{This research was supported by the ISRAEL SCIENCE FOUNDATION
	(grant No. 1167/17), 
	UNAM PREI Dgapa, 
	Unidad Mixta Internacional Laboratorio Solomon Lefschetz (LASOL), FONCICYT, 
	Papiit Dgapa UNAM IN106217, ECOS Nord-Conacyt 
	249542 
	and Conacyt 291231}

\date{\today}
	\begin{abstract}
		In this paper we study polynomial Hamiltonian systems $dF=0$ in the 
		plane and their small perturbations: $dF+\epsilon\omega=0$. The first 
		nonzero Melnikov function $M_{\mu}=M_{\mu}(F,\gamma,\omega)$ of the 
		Poincar\'e map along a loop $\gamma$ of $dF=0$ is given by an iterated 
		integral \cite{G}. In \cite{MNOP}, we bounded the length of the 
		iterated integral $M_\mu$ by a geometric number $k=k(F,\gamma)$ which 
		we call \emph{orbit depth}. 	
		We conjectured that the bound is 
		optimal. 
		
		Here, we give a simple example of a Hamiltonian system $F$ and its 
		orbit $\gamma$ having  infinite orbit depth. 
		If our conjecture is true, for this example there should exist 
		deformations $dF+\epsilon\omega$ with arbitrary high length first 
		nonzero Melnikov function $M_\mu$ along $\gamma$.	
		We construct deformations $dF+\epsilon\omega=0$ whose first nonzero 
		Melnikov function $M_\mu$ is of length three and explain the 
		difficulties in constructing deformations having high length first 
		nonzero Melnikov functions $M_\mu$. 
	\end{abstract}
	\maketitle
	
	\section{Introduction and Main Results}

This paper is motivated by two classical problems in the study of orbits of 
vector fields in the plane: the 16-th Hilbert problem and the center problem or 
rather their infinitesimal versions. 

The Infinitesimal Hilbert 16-th  problem asks for a bound on the number of 
limit 
cycles (i.e. isolated periodic orbits)  created by a small polynomial 
deformation  of a given degree of an integrable vector field in the plane. 

The infinitesimal center problem asks for a characterization of polynomial 
deformations of an integrable system which preserve  a family of loops. 

In both problems one studies the Poincar\'e first return map \eqref{P} on a 
transversal.
The first (possibly) nonzero term $M_\mu$ carries lots of information about the 
Poincar\'e map. Having an a priori estimate on its complexity would be very 
important for both infinitesimal problems. For the infinitesimal center problem,
to have an a priori estimate on the length is similar to having an estimate on 
the stabilization index for Noetherian property.

It is known \cite{F, G}, that when deforming a Hamiltonian vector field, the 
first nonzero  term $M_\mu$ is an iterated integral of length not exceeding its 
order $\mu$. However, the order  $\mu$ in general depends on the deformation. 
In \cite{MNOP}, we gave a bound on the length of $M_\mu$ by a geometric number 
\emph{orbit depth} $k$ which is independent on the deformation.   We showed 
that in different cases this bound is optimal and we conjectured that it is so 
in general.

\bigskip
In this paper we give an example where this bound is infinite. We believe that 
in the example one can construct deformations whose first nonzero Melnikov 
function $M_\mu$ is of arbitrarily high length. In that direction we construct 
for our example deformations having first nonzero Melnikov function $M_\mu$ of 
length $3$ and show the difficulties in constructing deformations with higher 
length. 

\begin{remark}
 \begin{myitemize}\hfill	
\item[(i)]	 Our example answers  negatively a question asked by Gavrilov and 
Iliev in \cite{GI}.
\item[(ii)]	 Our example shows the complexity of both infinitesimal problems. 
\end{myitemize}
\end{remark}
\bigskip

Let us be more 
precise. Let $F\in\C[x,y]$ be a polynomial and let  $\gamma\in \pi_1(F^{-1}(t))$ be a 
loop for $t$ a regular value of $F$.
Consider a small polynomial deformation
\begin{equation}\label{def}
dF+\epsilon\omega=0,
\end{equation}
of the Hamiltonian $dF=0$.
Let $\tau$ be a transversal section to $\gamma$ at a point $p_0$, parametrized 
by the values $t$ of $F$. Denote by  $P_\gamma$ the Poincar\'e return map 
(holonomy)  of 
\eqref{def} along $\gamma$. 
Then 
\begin{equation}\label{P}
P_\gamma(t)=t+\epsilon^\mu M_{\gamma,\mu}(t)+o(\epsilon^\mu).
\end{equation}

If the Poincar\'e map is not the  identity map, we assume that $M_\mu$ is 
nonzero and call it the \emph{first non-zero Melnikov function along $\gamma$ 
of the deformation \eqref{def}}. 

By the Poincar\'e-Pontryagin criterion, the first order  Melnikov function 
$M_1$ is given 
by an Abelian integral,
\begin{equation}
M_{\gamma,1}(t)=\int_\gamma\omega.
\end{equation}
More generally, $M_{\gamma,\mu}(t)$ is given as a linear combination of  iterated 
integrals of 
length at most $\mu$, see \cite{F, G}. However, this bound in general is not 
optimal.
For instance, for generic $F$ and any loop $\gamma$ and any deformation 
$\omega$, the first non-zero Melnikov functions $M_{\gamma,\mu}(t)$ is given by an 
Abelian 
integral (i.e. is an iterated integral of length $1$), irrespective of its 
order $\mu$. This follows from 
\cite{I,F}, see \cite{MNOP}. For other examples see \cite{MNOP}, as well as 
papers cited there. Moreover, the bound $\mu$, for the length of $M_{\gamma,\mu}$ 
depends on the deformation \eqref{def}.

In 
\cite{GI} a sufficent condition under which  the first nonzero Melinkov 
function $M_{\gamma,\mu}(t)$ is an Abelian integral
is formulated. We generalized this condition in \cite{MNOP}:
\medskip

Let $\Sigma$ be the set of atypical values of $F$, see \cite{HL}, and let 
$t\not\in\Sigma$ be some regular value of $F$. 
Denote $\Gamma_t=\{F^{-1}(t)\}$.
The fundamental group $\pi_1(\C\setminus\Sigma,t)$ 
acts on the fundamental group $\pi_1(\Gamma_t,p_0)$ as follows.  
For each generator $a_j$ of $\pi_1(\C\setminus\Sigma,t)$ corresponding to a 
closed 
curve $a_j(s)\subset\C\setminus\Sigma$, choose its  lifting $\tilde{a}_j(s)$, 
i.e. a loop $\tilde{a}_j(s)\subset 
F^{-1}(\C\setminus\Sigma)$ such that $F(\tilde{a}_j(s))=a_j(s)$ and 
$\tilde{a}_j(0)=\tilde{a}_j(1)=p_0$. Then, by Ehresmann's fibration theorem, 
the 
fundamental 
groups $\pi_1(F^{-1}({a}_j(s)), \tilde{a}_j(s))$ and $ 
\pi_1(F^{-1}({a}_j(s'))), 
\tilde{a}_j(s'))$ are canonically isomorphic for sufficiently close $s,s'$. 
This defines an automorphism $Mon(a_j)$ of $\pi_1(\Gamma_t,p_0)$ and the 
representation $Mon:\pi_1(\C\setminus\Sigma,t)\to Aut(\pi_1(\Gamma_t,p_0))$. 
This representation depends on the  choice of the liftings $\tilde{a}_j$, and 
different choices of liftings change $Mon(a_j)$ to conjugate automorphisms 
$\sigma_j^{-1} Mon(a_j)\sigma_j$, $\sigma_j\in\pi_1(\Gamma_t,p_0)$. We fix some 
choice of 
$\tilde{a}_j$.

\begin{definition}[see \cite{GI,MNOP}]
	Let $\OO$ be the smallest normal subgroup of $\pi_1(\Gamma_t,p_0)$ 
	containing the  orbit of $\gamma\in \pi_1(\Gamma_t,p_0)$ under 
	the action of $Mon(\pi_1(\C\setminus\Sigma,t)) $. Denote 
	$K=[\OO,\pi_1(\Gamma_t,p_0) ]$ and let $H_1(\OO)=\OO/K$.
\end{definition}
\begin{remark}
	Note that $\OO$, $K$ and $H_1(\OO)$ are independent on the particular 
	choice of $\tilde{a}_j$. Moreover, $H_1(\OO)$ is canonically isomorphic for 
	different  choices
	of $p_0$ in the following sense: the natural isomorphism  between 
	$\pi_1(\Gamma_t,p_0)$ and $\pi_1(\Gamma_t,p'_0)$ defined by a path joining 
	$p_0$ and $p'(0)$ descends to an isomorphism of the corresponding 
	$H_1(\OO)$, 
	and 
	this isomorphism is independent on the choice of this path.
\end{remark}

In what follows, we denote $\pi_1(\Gamma_t,p_0)$ by $\pi_1$. The \emph{lower 
	central sequence of $\pi_1$} is defined as:  
\begin{equation}
\pi_1=L_1\supset L_2=[L_1,\pi_1]\supset\cdots \supset 
L_{i+1}=[L_i,\pi_1]\supset\cdots
\end{equation}

There is a natural homomorphism $\iota_1:H_1(\OO)\to  H_1(\Gamma_t,\C)$, 
and let $\OO_1=\iota_1(H_1(\OO))=\frac{\OO L_2}{ 
	L_2}\otimes\C$. In general, $\iota_1$ is neither surjective 
nor 
injective. In \cite{GI} it is shown that if $\iota_1$ is injective then 
$M_\mu(t)$ is an Abelian integral.

In \cite{MNOP}, we defined the \emph{orbit depth} $k=k(F,\gamma)$,
\begin{definition} Given a polynomial $F\in\C[x,y]$ and a loop 
$\gamma\in\pi_1$ as above,  the orbit depth 
$k=k(F,\gamma)$ is defined as 
	\begin{equation}\label{k}
	k=\sup\,\left\{j\ge 1\,\left\vert\,\OO\cap 
		L_{j}\not\subseteq K\right.\right\}\subset\N\cup\left\{+\infty\right\}.
	\end{equation}
	
 We say that an element $v\in \OO$ is of depth $j$ if it belongs to $L_j$ and 
 its class 	in $H_1(\OO)$ is nonzero.
	Orbit depth is $k<\infty$ if $k$ is the highest depth of elements in $\OO$, 
	and it is infinite if there are elements of $\OO$ of arbitrary high depth. 
	
\end{definition}

In
\cite[Theorem 1.7]{MNOP}) we proved that the orbit depth $k=k(F,\gamma)$ 
bounds the 
length of iterated integrals representing the 
first nonzero Melnikov function $M_{\gamma,\mu}$ of small deformations 
\eqref{def}. 

We 
conjectured that it was an optimal 
bound for the length of the first nonzero Melnikov function $M_{\gamma,\mu}$ 
along 
$\gamma$ of deformations of $dF=0$. 
We hence believe that for a Hamiltonian system $dF=0$  and a loop 
$\gamma\in\pi_1(F^{-1}(t))$ of infinite orbit depth
there exist polynomial deformations \eqref{def} such that the first non-zero 
Melnikov function $M_{\gamma,\mu}$ is an iterated integral of arbitrary high 
length.

\begin{theorem}\label{thm:depth}
	There exists a polynomial function $F\in\R[x,y]$ and a loop 
	$\gamma\in\pi_1(F^{-1}(t))$
	such that the orbit depth $k$ of $\gamma$ is infinite.
	
	Such an example is given by
	\begin{equation}\label{F}	
	F(x,y)=(x^2-1)(y^2-1),
	\end{equation} and the loop $\gamma\subset 
	\{F=t\}$ 
	given by the real cycle vanishing at $(0,0)$ along the path 
	$(0,t)\subset\R$, for $t\in(0,1)$ (see Figure 1). 
\end{theorem}

\medskip
Our theorem also 
answers negatively to the question if $dim H_1(\OO)\le dim 
H_1(F^{-1}(t_0))$, which was raised as part of open question (1) in \cite{GI}.

\bigskip

We also prove 

\begin{theorem}\label{thm:M3} There exist a rational deformation 
$dF+\epsilon\omega$ of 
$F$ given by \eqref{F}
	such that the first nonzero Melnikov function $M_{\gamma,\mu}$ of the 
	deformation 
	\eqref{def} is an iterated integral of length $3$.
	An example of such deformation is a form $\omega$ of type
	\begin{equation}\label{eq:omega type}
	\omega=a_1(F)\frac{dx}{x+1}+ 
	a_2(F)\frac{dy}{y-1}+a_3(F)\frac{dx}{x-1},
	\end{equation}
	with $a_1(t)=t^2+2t, a_2(t)=t$ and $a_3(t)=t^2+t$.
	
	If $M_{\gamma,2}=M_{\gamma,3}\equiv0$ for deformation
	\eqref{def} with $\omega$ as \eqref{eq:omega type}, then the  deformation 
	is integrable.
\end{theorem}

We conclude that one needs a  richer set of 
deformations to get  an example of a perturbation with first nonzero Melnikov 
function $M_{\gamma,\mu}$ of length $\ge 4$.

\bigskip
One of the principal tools of the proof is Proposition~\ref{prop:BCHformula},
establishing connection between Poincar\'e  return maps of paths on $\Gamma_t$  
and the vector fields on the transversal $\tau$ whose flows give these maps.
\section{Example with infinite orbit depth}

We consider the polynomial $F(x,y)=(x^2-1)(y^2-1)$. The critical values of $F$ 
are $0$ and $1$, and the critical points are $(\pm 
1, \pm 1)$ on $\{F=0\}$ and $(0,0)$ at $\{F=1\}$.  Our goal in this section is 
to show that the orbit depth  of the real cycle $\gamma$ vanishing at the 
critical point $(0,0)$ is infinity.

The normalizations $\Gamma_t$ of complexifications of non-singular level curves  $\{F=t\}$, $t\not=0,1$, are 
torii  with $4$ points removed. 
The fundamental group $\pi_1(\Gamma_t, p_0)$ is a free group generated by loops 
$\gamma, \delta_0,\delta_1,\delta_2,\delta_3$, where $\delta_i$ are loops 
vanishing at $(\pm 1, \pm 1)$.

To be more precise, we take $0<t\ll 1$, choose $p_0$ close to the edge 
$\{x=-1\}$ of the square, and denote $\delta_0,\delta_1,\delta_2,\delta_3$ the  
geometric loops vanishing at $(-1,-1)$, $(1,-1)$, $(1,1)$ and $(-1,1)$ 
correspondingly, see the figure at \cite{MNOP}: we take a meridian of the 
cylinder which is $\{F=t\}$ near the corresponding singular point, with base 
point on $\gamma$,  and then pull the base point clockwise along $\gamma$ to 
$p_0$. We orient $\gamma$ counterclockwise, and orient $\delta_i$ in such a way that the intersection numbers $(\gamma, 
\delta_i)$ are all equal to one.

\begin{figure}[h]
	\begin{center}
		\includegraphics[height=5cm]{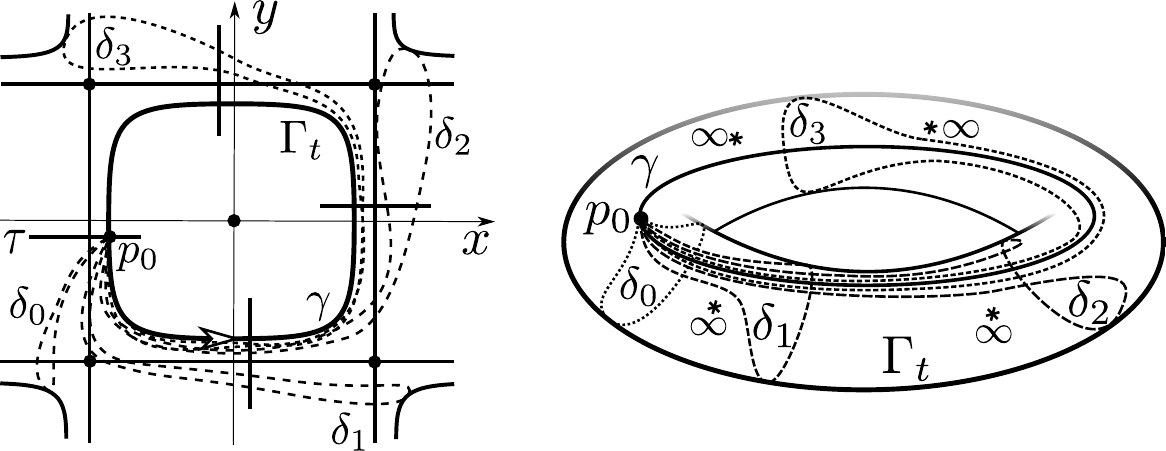}.
	\end{center}
	\caption{Generators of $\pi_1(\Gamma_t, p_0)$.}
\end{figure}
The atypical values of $H$ are exactly its critical values $0,1$. Therefore, 
the action of the monodromy of the 
foliation on the fundamental group of $\Gamma_t$ is generated 
by two automorphisms $Mon_0$ and $Mon_1$ of $\pi_1(\Gamma_t, p_0)$ 
corresponding to the loops going around the critical 
values $0$ and $1$ 
correspondingly, as described above.
\begin{figure}[h]
\begin{center}
	\includegraphics[height=2.5cm]{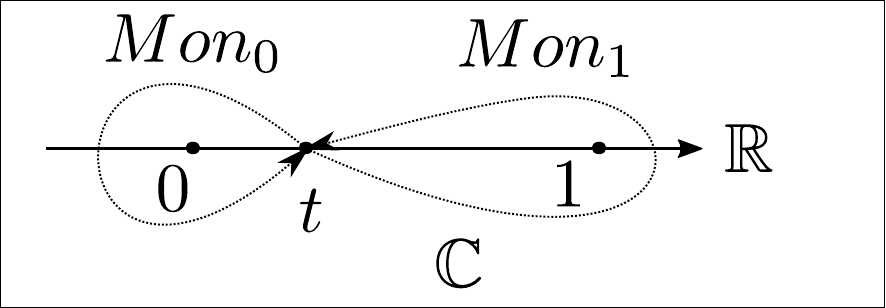}.
\end{center}
\caption{Generators of $\pi_1(\C\setminus\{0,1\}, t)$ corresponding to 
$Mon_0,Mon_1$.}
\end{figure}
\begin{lemma}\label{lem:monodromy}
 Denote $\delta=\delta_0\delta_1\delta_2\delta_3$. The monodromy operators 
 $Mon_{0,1}$ are	
\begin{align}
Mon_1=\{&\gamma\mapsto \gamma, \delta_i \mapsto \gamma\delta_i\}\\
Mon_0=\{&\gamma \mapsto\delta\gamma, \delta_0 \mapsto\delta_0, 
\nonumber\\
&\delta_1 \mapsto\delta_0\delta_1\delta_0^{-1}, \nonumber\\
&\delta_2 \mapsto\delta_0\delta_1\delta_2\delta_1^{-1}\delta_0^{-1},  
\nonumber\\
&\delta_3 
\mapsto\delta_0\delta_1\delta_2\delta_3\delta_2^{-1}\delta_1^{-1}\delta_0^{-1}\}
\end{align}
\end{lemma}

\begin{proof}
	These formulas follow from standard homotopical computations proving the 
	Picard-Lefschetz formula, see e.g. \cite{AVG2}. From the local topology in 
	a neighborhood of the center critical point, we have $Mon_1\gamma=\gamma$, 
	and $Mon_1\delta_i=\gamma\delta_i$, since the intersection number between 
	$\gamma$ and $\delta_i$ is one, and $\gamma$ is the cycle vanishing at the 
	center critical value when the regular value tends to 1. For the monodromy 
	around the critical value 0, we divide the real cycle $\gamma$ into pieces 
	$\gamma=\rho_0\rho_1\rho_2\rho_3$, where $\rho_{i}$ goes from a point $x_i$ 
	in $\gamma\cap U_i$ to a point $x_{i+1}$ in $\gamma\cap U_{i+1}$, where 
	$U_j$ is a neighborhood of the critical point at which $\delta_j$ vanish, 
	with $x_0$ and $x_4$ equal to the chosen initial point $p_0$.  Let 
	$\mathring{\delta_i}$ be the generator of $\pi_1(F^{-1}(t)\cap U_i,x_i)$. 
	From Picard-Lefschetz formula, locally at the neighborhood $U_i$ we have 
	$Mon_0(\rho_i)=\mathring{\delta_i}\rho_i$, and 
	$\Mon_0\mathring{\delta_i}=\mathring{\delta_i}$. Notice that 
	$\delta_i=\rho_0\cdots\rho_i\mathring{\delta_i}\rho_i^{-1}\cdots\rho_0^{-1}$.
	 Then, applying the local monodromy at each saddle critical point we get 
	the result. 
\end{proof}


\subsection{Chipping out Homology  and $Mon_1$}
Note that $\delta=Mon_0(\gamma)\gamma^{-1}$ is in $\OO$, and that 
$\gamma,\delta$ span the orbit of $\gamma$ in 
$H_1(\Gamma_t)$. From the exact sequence 
$$
0\to (\OO\cap L_2)/ K\to \OO/K\to \OO_1\to 0,$$
where $\OO_1=\langle \gamma, \delta\rangle\subset H_1(\Gamma_t)$, it is clear that the next step is to consider the 
action of the monodromy on $L_2=[\pi_1,\pi_1]$. It turns out that, up to subgroup generated by $\gamma$ and $\delta$, 
the action of $Mon_1$ on $L_2$ is trivial, thus allowing to disregard $Mon_1$.

Let $\Gamma$ be the normal subgroup of $\pi_1$ generated by $\gamma, \delta$. 
 Evidently, $\Gamma\subset\OO$ and $[\Gamma,\pi_1]\subset K$ is a normal subgroup of $\pi_1$
 generated by  commutators
$[\gamma, c]$, $[\delta,c], c\in\pi_1$.

The group 
$(\OO\cap L_2)/ K$ is  a 
subgroup of $L_2/K$, which is a factor of 
$L_2/[\Gamma,\pi_1]$. The latter is isomorphic to the commutator 
$G_2=[G,G]$ of the free group $G$ generated by $\delta_1,\delta_2,\delta_3$.

\begin{lemma}\label{lem:induced monodromy Mon_1}
	$Mon_1$ preserves both $L_2$ and $\Gamma$. The induced action of $Mon_1$ on 
	$\pi_1/\Gamma$ is 	trivial.  
\end{lemma}
\begin{proof}
	As $ Mon_1$ is an  automorphism of $\pi_1$, it preserves $L_2$.
	Also, $Mon_1(\gamma)=\gamma$, $Mon_1(\delta)=\gamma\delta_0\gamma\delta_1 
	\gamma\delta_2\gamma\delta_3=\delta \operatorname{mod}\Gamma\in\Gamma$, so 	
	$Mon_1$ preserves  $\Gamma$. Also 
	$Mon_1(\delta_i)=\gamma\delta_i=\delta_i \operatorname{mod}\Gamma$, which proves the last statement.
	\end{proof}

As $\Gamma\cap L_2\subset K$, this implies  that $Mon_1$ acts trivially on $ L_2/K$ and therefore can be disregarded.
\begin{corollary}
	$\OO$ is generated by $\gamma$ and $\Mon^i_0(\delta)$, $i=0,1,\dots$.
\end{corollary}

\begin{lemma}\label{lem:associated monodromy is trivial}
	$Mon_0$ preserves  $L_i\cap\langle 
	\delta_0, \delta_1,\delta_2, 
	\delta_3\rangle$ and 
	the induced action of $Mon_0$ on $L_i\cap\langle 
	\delta_0, \delta_1,\delta_2, 
	\delta_3\rangle /\left(L_{i+1}\cap\langle 
	\delta_0, \delta_1,\delta_2, 
	\delta_3\rangle\right)$ is trivial.
\end{lemma}
\begin{proof} Follows immediately from Lemma~\ref{lem:monodromy}.\end{proof}

So we have to investigate the orbit of $\delta$ in the free group generated by $\langle\delta_i,\,i=0,1,2,3\rangle$, 
under the 
action of $\Mon_0$ given by Lemma~\ref{lem:monodromy}.

Define $M(\sigma)=\delta_1^{-1}\delta_0^{-1}Mon_0(\sigma)\delta_0\delta_1$ on 
$\langle 
\delta, \delta_1,\delta_2, 
\delta_3\rangle$, 
\begin{equation}\label{eq:def of M}
M=\{\delta_0\mapsto\delta_1^{-1}\delta_0\delta_1, 
\delta_1\mapsto\delta_1,\delta_2\mapsto\delta_2,\delta_3\mapsto
[\delta_2,\delta_3]\delta_3\},  
\end{equation} 
and define $Var(\sigma)=M(\sigma)\sigma^{-1}$. Note that for $\sigma\in \OO$
\begin{equation}\label{eq:var and var_0}
Var(\sigma)=[\delta_1^{-1}\delta_0^{-1},Mon_0(\sigma)]Mon_0(\sigma)\sigma^{-1}
=Mon_0(\sigma)\sigma^{-1}\operatorname{mod}K\in\OO.
\end{equation}
This means that both $Mon_0$ and $M$ generate the same orbit, and one can use either of them. However, $M$ 
is computationally more convenient.
We formally define $Var(\gamma)=\delta$ and 
$Var^{i+1}(\gamma)=Var^i(\delta)$.

\begin{corollary}
	For any $i\ge 1$, $Var^i(\gamma)\in L_{i}\cap\langle 
	\delta, \delta_1,\delta_2, 
	\delta_3\rangle$.
\end{corollary}
\begin{proof} We have $Var(\gamma)=\delta\in \pi_1=L_1$. 
	Therefore $Var^i(\gamma)$ belong to the subgroup generated by $\delta_i$. 
	By 	induction, and using 
	Lemma~\ref{lem:associated monodromy is trivial} and \eqref{eq:var and var_0}, we see that  
	$Mon_0(Var^{i}(\gamma))=Var^i(\gamma)\operatorname{mod} L_{i+1}\cap\langle 
	\delta, \delta_1,\delta_2, 
	\delta_3\rangle$. Therefore $Var^{i+1}(\gamma)\in L_{i+1}\cap\langle 
	\delta, \delta_1,\delta_2, 
	\delta_3\rangle$.
\end{proof}

\begin{lemma}\label{lem:representation}
	Any element $w\in\OO $ can be represented as 
	\begin{equation}
	w=\gamma^{n_0}Var(\gamma)^{n_1}\dots Var^i(\gamma)^{n_k}\,\operatorname{mod} K, 
	\quad i=i(w).
	\end{equation}
\end{lemma}
\begin{proof}Indeed, any element in  $\OO/K$ is a product of $M^i(\gamma)$, and 
these 
elements can be represented in this form: if 
\begin{equation}
M^i(\gamma)=\gamma^{n_0}\delta^{n_1}\dots 
Var^i(\gamma)^{n_k}\,\operatorname{mod} K,
\end{equation}
then
\begin{align*}
M^{i+1}(\gamma)&=M(\gamma^{n_0})M(\delta^{n_1})\dots 
M(Var^i(\gamma)^{n_k})\,\operatorname{mod} K\\
&=(Var(\gamma)\gamma)^{n_0}(Var(\delta)\delta)^{n_1}\dots
(Var^{i+1}(\gamma)Var^i(\gamma))^{n_k}\,\operatorname{mod} K\\
&=\gamma^{n_0}\delta^{n_0+n_1}\dots Var^{i}(\gamma)^{n_{k-1}+n_k}
Var^{i+1}(\gamma)^{n_k}\,\operatorname{mod} K,
\end{align*}
as  $Var^i(\gamma)$
commute modulo $K$.
\end{proof}
\bigskip

Define by induction the maps  $d_k:\pi_1\to\pi_1$   as  
$d_1=Id$, and
$d_{k+1}(\sigma)=[\delta_2,d_k(\sigma)]$. Note that
$d_k(\sigma)=[\delta_2,[\delta_2,[\dots[\delta_2,\sigma]\dots]\in
L_k$ for all $k\ge 1$, $\sigma\in\pi_1$. 

\begin{proposition}\label{prop:variations}
Denote  $x=\delta_1\delta_2$, $z=\delta_2\delta_3$ and define 
\begin{equation}
v_1=\delta,\quad	
v_k=[x,d_{k-1}(z)] \text{ for } k\ge2.
	\end{equation}
Then $v_i\in\OO$ and 
	$Var^i(\gamma)=v_i\operatorname{mod}K$.
\end{proposition}

Before proving Proposition~\ref{prop:variations}, we prove
\begin{lemma} 
	$d_{k-1}([\delta_2,z]z)=d_k(z)d_{k-1}(z)$.
\end{lemma}
\begin{proof} By induction,
	\begin{align}
	d_k([\delta_2,z]z)&=[\delta_2, d_{k-1}([\delta_2,z]z)]=[\delta_2, 
	d_{k}(z)d_{k-1}(z)]=\nonumber\\
	&=[\delta_2, d_k(z)][\delta_2, d_{k-1}(z)]\,\big[[ 
	d_{k-1}(z),\delta_2], d_k(z)\big]=\nonumber\\
	&=d_{k+1}(z)d_k(z)[ d_k(z)^{-1}, d_k(z)]=d_{k+1}(z)d_k(z).\qedhere
	\end{align}
\end{proof}

\begin{proof}[Proof of Proposition~\ref{prop:variations}]
From Lemma~\ref{lem:monodromy} we see that 
$v_1=Mon_0(\gamma)\gamma^{-1}=\delta$.

Note that by \eqref{eq:def of M}
\begin{equation}\label{eq:M of x,delta2 z}
M(x)=x,\quad M(\delta_2)=\delta_2,\quad 
M(z)=\delta_2z\delta_2^{-1}=[\delta_2,z]z.
\end{equation} 

We have 
\begin{equation}
\delta=\delta_1\delta_2\delta_3\delta_0[\delta_0^{-1},\delta^{-1}]=
\delta_1\delta_2\delta_3\delta_0\,\operatorname{mod}K,
\end{equation} so, modulo $K$,
\begin{align}
Var^2(\gamma)&=M(\delta_1\delta_2\delta_3\delta_0)\delta^{-1}=
\delta_1\cdot\delta_2\cdot\delta_2\delta_3\delta_2^{-1}
\cdot\delta_1^{-1}\delta_0\delta_1\cdot\delta^{-1}=\nonumber\\
&=[\delta_1\delta_2,\delta_2\delta_3][\delta_2\delta_3,\delta]=[x,z]=v_2.
\end{align}
In particular, $v_2\in\OO$.

For the third variation,  again modulo $K$,
\begin{align}
Var^3(\gamma)&=M([x,z])v_2^{-1}=[x,[\delta_2,z]z]v_2^{-1}=\nonumber\\
&=[x,[\delta_2,z]][x,z][[z,x],[\delta_2,z]]v_2^{-1}=[x,[\delta_2,z]].
\end{align}
as $[[z,x],[\delta_2,z]]=[v_2^{-1},[\delta_2,z]]\in K$.

Now, from \eqref{eq:M of x,delta2 z} follows  
$
M(v_{k})=[x,d_{k-1}([\delta_2,z]z)],
$ so 
modulo $K$,
\begin{align}
Var^{k+1}(\gamma)&=M(v_{k})v_{k}^{-1}=[x,d_{k}(z)d_{k-1}(z)]v_k^{-1}=\nonumber\\
&=[x,d_{k}(z)][x,d_{k-1}(z)]\,\big[[d_{k-1}(z),x],d_{k}(z)\big]v_k^{-1}=\nonumber\\
&=[x,d_{k}(z)]\,v_k\,[v_k^{-1},d_{k}(z)]\,v_k^{-1}=[x,d_{k}(z)]=v_{k+1}.\qedhere
\end{align}
\end{proof}

\begin{proposition}\label{prop:generators of OO}
$\OO=\langle\gamma, v_i, 
	i=1,...\rangle$.
\end{proposition}
\begin{proof} Denote $\OO^v=\langle\gamma, v_i, 
	i=1,...\rangle$ the normal subgroup of $\pi_1$ generated by $\gamma, v_i$. 
	It follows from Proposition~\ref{prop:variations} that 
	$v_i\in\OO$, so $\OO^v\subset\OO$. 
	
	Let us prove the opposite inclusion. By Lemma~\ref{lem:representation}, $\OO=\langle \gamma, Var^i(\gamma), 
	i=1,\dots\rangle$. 
	By Lemma~\ref{prop:variations},
	\begin{equation}
	Var^i(\gamma)=v_i W_{i,1}(\{[a^1_j,Var^j(\gamma)]\}_{j\ge 0}), \quad i\ge 0,
	\end{equation} where $W_{i,1}$ are some words. Substituting  these 
	equalities into their right hand sides, we get 
\begin{equation}
Var^i(\gamma)=v_i W_{i,2}(\{[b_j,v_j],[a^2_j,[a^1_j,Var^j(\gamma)]]\}_{j\ge 
0}), \quad i\ge 0.
\end{equation}
Repeating substitution $\ell-1$ times, we get for any $\ell\ge 1$
\begin{equation}
Var^i(\gamma)=v_i\epsilon_{i,\ell} w_{i,\ell}, \quad 
\epsilon_{i,\ell}\in\OO^v,\,\,
w_{i,\ell}\in [\pi_1,[\pi_1,[\dots[\pi_1,\OO]\dots]\subset L_{\ell+1}.
\end{equation}
This implies that $\OO L_{\ell+1}\subset\OO^v L_{\ell+1}$ for any $\ell\ge 1$.
 This 
implies that $\OO\subset\OO^v$.
\end{proof}

\begin{corollary}\label{cor:generators of K}
	$K=\langle [\pi_1,\gamma],[\pi_1, v_i], 
	i=1,...\rangle$.
\end{corollary}

\subsection{Depth is infinite}

Here we prove Theorem~\ref{thm:depth}. The main idea is to construct for any $k$ a matrix representation
$\rho_k$ of $\pi_1$ sending all generators of $\OO$ except $v_{k+2}$ to 
identity. We prove that 
$\rho_k(v_{k+2})\not\in\rho_k(K)$ 
for a generic choice of parameters $a,c$ of the representation $\rho_k$, which implies Theorem~\ref{thm:depth}.

Let $A_0=a$, $B_0=1$ and $C_0=c$, where $a,c\not=
 0,1$.
Define inductively the $2^k\times 2^k$-matrices $A_k,B_k,C_k$ as follows:
\begin{equation}
A_{k+1}=\begin{bmatrix}
A_k& 0\\
0& \mathbb{I}
\end{bmatrix},
B_{k+1}=\begin{bmatrix}
B_k& \mathbb{I}\\
0& B_k
\end{bmatrix},
C_{k+1}=\begin{bmatrix}
\mathbb{I}& 0\\
0& C_k
\end{bmatrix},
\end{equation}
where $\mathbb{I}$ is the corresponding identity matrix.

\begin{proposition}\label{prop:representationABC}
	Let $x=\delta_1\delta_2$ and $z=\delta_2\delta_3$ be as in 
	Porposition~\ref{prop:variations}.
	Consider the representation $\rho_k:\pi_1\to GL(2^k)$ defined by 
	$\rho_k(\gamma)=\rho_k(\delta)=\mathbb{I}, \rho_k(x)=A_k, 
	\rho_k(\delta_2)=B_k$ 
	and $\rho_k(z)=C_k$. Then 
	\begin{enumerate}
		\item $\rho_k(v_i)=\mathbb{I}$ for $i\not= k+2$ and 
		$\rho_k(v_{k+2})\not=\mathbb{I}$, and
		\item $\rho_k(v_{k+2})\notin 
		[\rho_k(v_{k+2}),\rho_k(\pi_1)]=\rho_k(K)$ for generic $a,c$.
	\end{enumerate}
\end{proposition}

\begin{remark}
	As $\pi_1$ is a free group and $\gamma,\delta,x,\delta_2,z$ are its 
	generators, such a representation $\rho_k$ exists.
\end{remark}

\begin{proof}

We start with  another description of $A_k,B_k$ and $C_k$. Let
$$\I_2=\begin{bmatrix}
1& 0\\
0& 1
\end{bmatrix},\quad
J_2=\begin{bmatrix}
0& 1\\
0& 0
\end{bmatrix},\quad
 E_2=\begin{bmatrix}
0& 0\\
0& 1
\end{bmatrix},\quad
F_2=\begin{bmatrix}
1& 0\\
0& 0
\end{bmatrix}.
$$
Recall that tensor products are 
multiplied factorwise:
$$
(X_1\otimes Y_1)(X_2\otimes Y_2)=X_1X_2\otimes Y_1Y_2.
$$

Denote  by $b_{j_1\dots j_l}$ to be  the tensor product of  $(k-l)$ 
copies of $\I_2$ and  $l$ copies of
$J_2$, with $J_2$ being exactly the $j_1^{\text{th}},\dots,j_l^{\text{th}}$ 
factors. Similarly, denote  by $e_{j_1\dots j_l}$ to be  the tensor product of  
$(k-l)$ 
copies of $E_2$ and  $l$ copies of
$J_2$, with $J_2$ being exactly the $j_1^{\text{th}},\dots,j_l^{\text{th}}$ 
factors.
Finally, denote  $\alpha=F_2^{\otimes k}$, $\gamma=E_2^{\otimes k}$ and 
$\beta=\sum_{j=1}^k b_j$.

Using this notations, we have
\begin{equation}
A_k=\I_{2^k}+(a-1)\alpha,\quad B_k=\I_{2^k}+\beta
\quad
\text{and} \quad C_k=\I_{2^k}+(c-1)\gamma.
\end{equation}

Our immediate goal is to compute $[B_k,C_k]$. 
Evidently,
\begin{equation}
A_k^{-1}=\I_{2^k}+(\tfrac 1 a -1)\alpha, \quad C_k^{-1}=\I_{2^k}+(\tfrac 1 c 
-1)\gamma. 
\end{equation}	
	
	As $J_2^2=0$, we have $b_{j_1\dots j_l}b_{j'_1\dots j'_{l'}}$ equals 
	$b_{j_1\dots j_lj'_1\dots j'_{l'}}$ if the sets $\{j_1\dots 
	j_l\}$,$\{j'_1\dots j'_{l'}\}$ do not intersect,  and zero otherwise.
	
	Therefore	
	\begin{equation}	\beta^l=l!\sum_{1\le j_1<\dots<j_l\le k} b_{j_1\dots 
		j_l}, \quad\beta^k=k!J_2^{\otimes k}\quad\text{and}\quad 
		\beta^{k+1}=0.                    
	\end{equation}
	In particular, 
	$$
	B_k^{-1}=\I_{2^k}+\tilde{\beta},\quad\text{where}\quad\tilde{\beta}=
	-\beta+\beta^2 	-\dots\pm\beta^k.
	$$
	
	Now, from $N_2^2=N_2, N_2J_2=0, J_2N_2=J_2$
		we have
	\begin{equation}
	\gamma^2=\gamma,\quad \gamma\beta^l=\gamma\tilde{\beta}=0,
	\end{equation}
	and
	\begin{equation}
	\beta^l\gamma=\epsilon^{[l]}=l!\sum_{1\le j_1<\dots<j_l\le k} 
	e_{j_1\dots 		j_l},
	\end{equation}

	Note that $\epsilon^{[l]}\epsilon^{[l']}=0$ for all $l,l'\ge 1$,
	$\epsilon^{[k]}=k!J_2^{\otimes k}$  and $\epsilon^{[l]}=0$ for $l>k$.

	Now, using the above formulae we see that 
	\begin{align}\label{eq:v_2}
	[B_k,C_k]=&(\I_{2^k}+\beta)(\I_{2^k}+(c-1)\gamma) 
	(\I_{2^k}+\tilde{\beta})(\I_{2^k}+(\tfrac 1 c 
	-1)\gamma)\nonumber\\
=	&
	\I_{2^k}-(\tfrac 1 c -1)\beta\gamma=\I_{2^k}-(\tfrac 1 c 
	-1)\epsilon^{[1]},
	\end{align}
	and, as $\left(\epsilon^{[1]}\right)^2=0$, 
	\begin{equation}
	[B_k,C_k]^{-1}=\I_{2^k}+(\tfrac 1 c 
	-1)\epsilon^{[1]}.
	\end{equation}
	Continuing,
	\begin{equation}
	[B_k,[B_k,C_k]]=\I_{2^k}-(\tfrac 1 c 
	-1)\beta\epsilon^{[1]}=\I_{2^k}-(\tfrac 
	1 c 
	-1)\epsilon^{[2]},
	\end{equation}
	and, by induction, for a commutator with $l$ entries of $B_k$,
	\begin{equation}
	[B_k,[\dots[B_k,C_k]]\dots]=\I_{2^k}-(\tfrac 1 c 
	-1)\epsilon^{[l]}.
	\end{equation}
	Now,
	similarly, from 
	\begin{equation}
	F_2^2=F_2,\quad F_2J_2=J_2,\quad  J_2F_2=0
	\end{equation}
	we have
	\begin{align*}
	\alpha^2&=\alpha,
	&&\epsilon^{[l]}\alpha=0  \text{ for all } l,\\
\alpha\epsilon^{[k]}&=\epsilon^{[k]}=k!J_2^{\otimes k},	&&\alpha\epsilon^{[l]}=0
 \text{ for }  l\neq k.
		\end{align*}
		
	Therefore
	\begin{equation}
	\rho_k(v_{l+2})=[A_k,[B_k,[\dots[B_k,C_k]]\dots]]=\I_{2^k}+(\tfrac 1 
	c-1)(\tfrac 1 
	a 
	-1)\alpha\epsilon^{[l]},
	\end{equation}
	i.e.
\begin{align}\label{eq:rho_k(v_l)}
	\rho_k(v_{l+2})&=\I_{2^k}\qquad\qquad \text{ for   } \quad l+2\neq k,& 
	\nonumber\\
	\rho_k(v_{k+2})&=\I_{2^k} +(\tfrac 1 c-1)(\tfrac 
	1 a 
	-1)k!J_2^{\otimes k}\neq \I_{2^k},&
\end{align}
	which proves the first claim of Proposition~\ref{prop:representationABC}.
	
	Let $s=\prod g_i^{m_i}\in\pi_1,$ where $g_i\in\{\gamma,\delta,x,\delta_2,z\}$ and $m_i\in\Z$.
	Then $\rho_k(s)=D+U$, where $U$ is a strictly upper triangular matrix and 
	$D=\operatorname{diag}(a^m,1,\dots,1,c^n)$.
	Therefore 
	\begin{equation}\label{eq:comm of repr}
	[\rho_k(s),\rho_k(v_{k+2})]=\I_{2^k}+\left(\tfrac{a^m}{c^{n}}-1\right)(\tfrac 1 c-1)(\tfrac 
	1 a 
	-1)k!J_2^{\otimes k}.
	\end{equation}
	
	Now, assume \begin{equation}\label{eq:v_k+2 impossible}
	\rho_k(v_{k+2})=\prod_j [\rho_k(s_j),\rho_k(v_{k+2})].
	\end{equation} By \eqref{eq:rho_k(v_l)},\eqref{eq:comm of 
	repr}, 
	we have
	\begin{equation}
\I_{2^k} +(\tfrac 1 c-1)(\tfrac 1 a -1)k!J_2^{\otimes k}=
\I_{2^k} +\left(\sum\left(\tfrac{a^{m_j}}{c^{n_j}}-1\right)\right)
(\tfrac 1 c-1)(\tfrac 1 a -1)k!J_2^{\otimes k},
	\end{equation}
	or, equivalently, $1=\sum\left(\tfrac{a^{m_j}}{c^{n_j}}-1\right).$
Collecting similar terms, we  get 
$$ 
\sum\lambda_i\tfrac{a^{m_i}}{c^{n_i}}=1+\sum\lambda_i, \quad\text{where}\quad\lambda_i,m_i, n_i\in\Z,$$
and for any $i$ one of the exponents $m_i,n_i$ is non-zero.  
This cannot hold for all 
$a,c$: if the left hand side is a constant, then all $\lambda_i$ vanish, and we get $0=1$. 
Therefore  any representation \eqref{eq:v_k+2 impossible} fails on a Zariski 
open subset of $\C_{(a,c)}$, so all such 
representations fail for a generic choice of $a,c$.
\end{proof}

\begin{proof}[Proof of Theorem~\ref{thm:depth}] By Corollary~\ref{cor:generators of K} and 
Proposition~\ref{prop:representationABC}(1), we have 
	$\rho_k(K)=[\rho_k(\pi_1),\rho_k(v_{k+2})]$.  By Proposition~\ref{prop:representationABC}(2), 
	$\rho_k(v_{k+2})\notin[\rho_k(\pi_1),\rho_k(v_{k+2})]$ for a generic choice of $a,c$. This means that 
	$\left(\OO\cap L_{k+2}\right)\setminus K$ contains $v_{k+2}$, so is non-empty for all $k$.
\end{proof}

\section{First nonzero Melnikov function of length 3}
\subsection{Cohomologies: notations.}
Denote $f_1=x+1$, $f_2=y-1$, $f_3=x-1$ and $f_4=y+1$. Denote  $\phi_i=\log f_i$ and $\eta_i=d\phi_i=\tfrac{df_i}{f_i}$.

The cycles $\gamma,\delta, \delta_1,\delta_2,\delta_3$ form a basis of 
$H_1(\Gamma_t)$, and $\gamma,\delta$ form 
a basis of the orbit of $\gamma$ in $H^1(\Gamma_t)$.  
As $ \phi_i$ are univalued on $\gamma$, the restrictions to $\Gamma_t$ of polynomial forms $\{F\eta_i\}_{i=1}^3$ 
lie in the orthogonal complement $\OO^\perp\subset H^1(\Gamma_t)$ of the orbit $\OO_1\subset H_1(\Gamma_t)$ of $\gamma$ 
in $H_1(\Gamma_t)$, and in fact form its basis. We have 
\begin{equation}
\begin{array}{lll}
\int_{\delta_1}\eta_1=0,&\int_{\delta_1}\eta_2=0,&\int_{\delta_1}\eta_3=2\pi i 
\\ 
\int_{\delta_2}\eta_1=0,&\int_{\delta_2}\eta_2=2\pi  i,
&\int_{\delta_2}\eta_3=-2\pi i\\
\int_{\delta_3}\eta_1=2\pi i,
&\int_{\delta_3}\eta_2=-2\pi i,&\int_{\delta_3}\eta_3=0.
\end{array}\label{eq:delta eta pairing}
\end{equation}
Note that $d\eta_i=0$, so the Gelfand-Leray
derivatives $\frac{d\eta_i}{dF}$ vanish.     

\subsection{Linear perturbations}
Consider the rational 1-form of type \eqref{eq:omega type}, i.e.
\begin{equation}\label{omega}
\omega=a_1(F){\eta_1}+a_2(F){\eta_2}+a_3(F){\eta_3},
\end{equation}
where $a_i(t)$ are  holomorphic on $\tau$, and consider 
the perturbation 
\begin{equation}
dF+\epsilon\omega=0,\quad F=(x^2-1)(y^2-1).
\end{equation}

\begin{remark}
	Note that restriction to $\Gamma_t$ of any form $\omega$  such that 
	$\int_{\gamma(t)}\omega\equiv0$ is cohomologous to a linear combination of 
	$\eta_i$. 
\end{remark}

The Poincar\'e map along the cycles $\gamma$ is 
\begin{equation}
P_{\gamma}(t)=t+\epsilon M_{\gamma,1}(t)+\epsilon^2M_{\gamma,2}(t)
+\epsilon^3M_{\gamma,3}(t)+\cdots,
\end{equation}
with $M_{\gamma,1}(t))=\int_{\gamma(t)}\omega\equiv 0$. Our goal is to find a 
polynomial form $\omega$
providing the highest possible order of the first non-vanishing Melnikov 
function $M_{\gamma,i}(t)$
\begin{proposition}\label{prop:pert3}
$M_{\gamma,2}\equiv 0$ and $M_{v_3,3}\not\equiv 0$ if, and only if,
\begin{equation}
\begin{aligned}
a_3(t)&=\alpha_1\int_{0}^{t}\frac{\alpha_2(\tau)}{\alpha_1^2(\tau)}d\tau
+c_0\alpha_1\\
a_1(t)&=a_3(t)+\alpha_1(t)\\
a_2(t)&=\lambda\alpha_1, 
\end{aligned}\label{thm:ConditionsM2(v_2)=0,M_3(v_3)not0}
\end{equation}
where $\lambda\in\C^*$ and $\alpha_1(t)$ and $\alpha_2(t)$ are linearly 
independent functions over 
$\C$, and $\alpha_1$ is not constant.

\end{proposition}

To prove Proposition~\ref{prop:pert3}, we consider the second and third 
variations of $\gamma$, i.e. $v_2$ and $v_3$ from Proposition 
\ref{prop:variations}, and the corresponding Poincar\'e maps $P_{v_2}$ and 
$P_{v_3}$. Proposition~\ref{prop:BCHformula} implies that
\begin{equation}\label{Poincaremap:P_{v_i}}
\begin{array}{rl}
P_{v_2}(t)&=t+\epsilon^2M_{v_2,2}(t)+O(\epsilon^3),\\
P_{v_3}(t)&=t+\epsilon^3M_{v_3,3}(t)+O(\epsilon^4),
\end{array} 
\end{equation}
and provides explicit expression of $M_{v_2,2}(t), M_{v_3,3}(t)$ in terms of 
coefficients $a_i(F)$. This allows to find  conditions on $a_i$ guaranteeing 
$M_{v_3,3}(t)\not\equiv 0$ and 
$M_{v_2,2}(t)\equiv 0$. We prove that the last condition is equivalent to 
$M_{\gamma,2}(t)\equiv 0$ in Lemma~\ref{lem:M22 equivalent Mgamma2}.

\begin{remark}[Geometric interpretation of Proposition~\ref{prop:pert3}]
	The forms \eqref{omega} form a
	three dimensional module 
	$\Omega$ over the ring of germs of holomorphic functions at $t$. The  
	Poincar\'e map along $\gamma$ is a map 
	$\mathcal{P}_\gamma:U\subset\Omega\to \mathcal{H}ol(\tau)$, 
	where $\mathcal{H}ol(\tau)$ is the set of germs of holomorphic mappings 
	$g:(\tau,p_0)\to \tau$. 
	
    The perturbations \eqref{def} are germs of lines in 
	$\Omega$, and the order of the first non-zero 
	Melnikov function of the perturbation can be interpreted as  the order of 
	vanishing of 
	$\mathcal{P}$ on these lines, i.e.  the 
	order of tangency of these lines to the set $\{\mathcal{R}=0\}$ of 
	integrable perturbations. 
	Theorem~\ref{thm:M3} claims that the maximum order of this tangency 
	is either at most three or the line lies 
	entirely in   $\{\mathcal{R}=0\}$.
	
	To construct the perturbations with first non-zero Melnikov function $M_k$ 
	of higher length, we necessarily have to 
	increase $k$, i.e. the order of tangency of the perturbation with the set of integrable foliations. This means that 
	we have either to consider non-linear perturbations, i.e. germs 
	of curves in $\Omega$, or consider a wider class $\tilde{\Omega}$ of 
	perturbations, e.g. by including relatively exact forms. 
	
	Still,  the first non-zero Melnikov function of a non-linear perturbation 	
	\begin{equation}
	dF+\epsilon\omega_1+\epsilon^2\omega_2+...=0,\quad \omega_i\in\Omega
	\end{equation}
	can be of high order, but of small length.  
	It is easy to see that  the terms of highest length of the corresponding 
	Melnikov functions depend only on $\omega_1$.  Thus, to 
	ensure that the length of the first non-vanishing Melnikov functions is at 
	least $4$, we should take $\omega_1$ such 
	that $dF+\epsilon\omega_1=0$ is integrable (otherwise $M_3\neq0$), and find non-linear terms in such a way  that 
	$M_{\gamma,4}\equiv 0$ (as its longest terms are determined by $\omega_1$, 
	they 
	necessarily vanish, so its length could be at most $3$), but 
	$M_{\gamma,5}\neq0$ and 
	has length $4$ (it 	cannot be of length $5$ by the same reason). The 
	latter would follow from 
	$M_{v_4,5}\neq0$. This program can be realized,  but it is computationally 
	hard. Moreover, it is not clear how one can generalize this approach to 
	higher length, so we omit the computations.
\end{remark}

\subsection{Poincar\'e maps as time-one flows of  vector fields on the 
transversal}
Consider the family \eqref{def} as a one-dimensional foliation \begin{equation}
\mathcal{F}=\{dF+\epsilon\omega=0, \, d\epsilon=0\}
\end{equation}
in $\C^3_{x,y,\epsilon}$. Let $\Theta=\tau\times (\C_\epsilon,0)$ be a 
transversal to the algebraic leaf $\Gamma=\Gamma_t\times \{0\}$ at the point 
$(p_0,0)$, and denote $\mathcal{D}=Diff((\Theta,(p_0,0)))$ the group of germs 
of 
holomorphic diffeomorphisms of $\Theta$. Holonomy of $\mathcal{F}$ along 
various 
paths $\gamma\in\pi_1(\Gamma, (p_0,0))$  defines a representation 
$\tilde{P}:\pi_1\to \mathcal{D}$ 
preserving $\epsilon$, i.e.  
$\tilde{P}_\gamma:(x,\epsilon)\to(P_\gamma(x,\epsilon),\epsilon)$ for any 
$\gamma\in\pi_1$.

Define $v_\gamma=(d\tilde{P}_\gamma)(\partial_\epsilon)$, 
$v_e=\partial_\epsilon$, and let $\phi_\gamma^s$ be the $s$-time flow of 
$v_\gamma$ 
(necessarily $L_{v_\gamma}(\epsilon)=1$). By definition, $\tilde{P}_\gamma$ 
conjugates flows of $v_e$ and $v_\gamma$. In particular, for all $p\in\tau$
\begin{equation}
\tilde{P}_\gamma(p,\epsilon)=\tilde{P}_\gamma(\phi_e^\epsilon(p,0))
=\phi_\gamma^\epsilon(\tilde{P}_\gamma(p,0))=\phi_\gamma^\epsilon(p,0),
\end{equation}
as $\tilde{P}_\gamma(p,0)\equiv(p,0)$ for all $\gamma\in\pi_1$.

Let $(t=F(p), \epsilon)$ be parameterization of $\Theta$. The expansion 
\eqref{P} is the expansion of $\tilde{P}_\gamma$ in degrees of $\epsilon$,
\begin{equation}
\tilde{P}_\gamma(t,\epsilon)=(t+\epsilon^\mu 
M_{\gamma,\mu}(t)+o(\epsilon^\mu),\epsilon). 
\end{equation}
Let 
\begin{equation}
v_\gamma=(v_\gamma^0(t)+\epsilon v_\gamma^1(t)+...)\partial t+\partial_\epsilon
\end{equation}
be decomposition of $v_\gamma$. Evidently, 
$v_\gamma^0=\dots=v_\gamma^{\mu-2}\equiv0$, and 
$v_\gamma^{\mu-1}(t)=M_{\gamma,\mu}(t)$.
\begin{proposition}\label{prop:BCHformula}
	Let $\gamma_1,\gamma_2,\gamma=[\gamma_1,\gamma_2]\in\pi_1$, and let 
	$$
	P_{\gamma_i}(t,\epsilon)=t+\epsilon^{\mu_i}M_{\gamma_i,\mu_i}+o(\epsilon^{\mu_i}),
	\,\,i=1,2,$$ with $M_{\gamma_1,\mu_1},M_{\gamma_2,\mu_2}\not\equiv0$. 
	
	Then 
	$v_\gamma^0=\dots=v_\gamma^{\mu-2}\equiv0$ for $\mu=\mu_1+\mu_2$, and 
\begin{equation}
	v_{[\gamma_1,\gamma_2]}^{\mu-1}(t)=W(M_{\gamma_1,\mu_1}(t),M_{\gamma_2,\mu_2}(t)),
\end{equation} where 
	$W(f,g)=fg'-f'g$ denotes the Wronskian of $f,g$. 
\end{proposition}
	Alternatively, 
	\begin{equation}
	v_{[\gamma_1,\gamma_2]}^{\mu-1}(t)\partial_t=\left[M_{\gamma_1,\mu_1}(t)\partial_t,
	M_{\gamma_2,\mu_2}(t)\partial_t\right],
	\end{equation}
	where brackets denote the Lie bracket of vector fields.

\begin{remark}
	Essentially, \eqref{def} induces a homomorphism $\mathfrak{R}$ of the 
	fundamental group 
	$\pi_1$ of $\Gamma_t$ to the group of germs at identity of analytic curves 
	in the groupoid $Diff(\tau)$. 
	
	More precisely,  for any $\gamma\in\pi_1$ we get a germ 
	$\mathfrak{R}(\gamma)$ 
	at identity of an 
	analytic curve 
	$$\left\{\hat{\omega}_\gamma\right\}=
	\left\{\tilde{P}_\gamma(\cdot,\epsilon)\right\}\subset Diff(\tau).$$
 The Lie algebra $\mathcal{X}$ of $Diff(\tau)$ "is" the Lie 
	algebra of germs at $p_0$ of vector fields on $\tau$, and $v_\gamma$ 
	defines  the corresponding (under exponential map)  path 
	$\log\hat{\omega}_\gamma$  in this Lie algebra. The path 
	$\tilde{P}_\gamma(\cdot,\epsilon)$ is not necessarily a one-parametric 
	group, and the path $\log\hat\omega_\gamma$ is not necessarily constant, 
	but we  are interested in the leading term of $\hat\omega_\gamma$ 
	only. If $\mu=1$, then the leading term is  the tangent vector 
	$M_{\gamma,1}\partial_t$ to $\hat\omega_\gamma$. However, if $\mu>1$ then 
	the 
	tangent vector is zero.

	To include the case $\mu>1$ consider  the group $\mathfrak{G}$ of germs at 
identity of 
analytic 	curves 	in the groupoid $Diff(\tau)$. $\mathfrak{G}$ has natural 
filtration by order of tangency 
of the germ to the constant germ, i.e. by the order of the 
first non-zero term in its Taylor decomposition in $\epsilon$.
 which induces a filtration on its Lie 
algebra, and  the associated 
graded algebra  is  a Lie algebra $\hat{\mathfrak{G}}$ isomorphic to 
$\mathcal{X}\otimes 
\C[[\epsilon]]$, up to a shift of grading by $1$. The homomorphism 
$\mathfrak{R}$ pulls back the above filtration of $\mathfrak{G}$ to a 
filtration of $\pi_1$, 
compatible with the group commutator (for generic perturbations this 
filtration most probably coincides with the lower central series $L_i$). 
Starting from this filtration on $\pi_1$, one can build a Lie algebra in a 
standard way,
and Proposition~\ref{prop:BCHformula} shows that $\mathfrak{R}$ lifts to a Lie 
algebra mapping between this Lie algebra 
and  $\hat{\mathfrak{G}}$.	
\end{remark}

\begin{proof}
The monodromy of $\gamma=\gamma_1\gamma_2\gamma_1^{-1}\gamma_2^{-1}$ is given 
by $P_\gamma=P_{\gamma_2}^{-1}\circ	P_{\gamma_1}^{-1}	\circ 
P_{\gamma_2}\circ P_{\gamma_1}$. 	
Denote 
$P_{\gamma_i}(t,\epsilon)=t+\epsilon^{\mu_i}M_{\gamma_i,\mu_i}+\dots+\epsilon^\mu
 M_{\gamma_i,\mu}+o(\epsilon^{\mu})$ for $i=1,2$.
Then
\begin{align}
P_{\gamma_1}\circ P_{\gamma_2}(t,\epsilon)&=P_{\gamma_2}(t,\epsilon)
+\sum_{j=\mu_1}^{\mu}\epsilon^jM_{\gamma_1,j}(P_{\gamma_2}(t,\epsilon))
+o(\epsilon^{\mu})=\nonumber\\
&=t
+\sum_{j=\mu_2}^{\mu}\epsilon^jM_{\gamma_2,j}(t)
+\sum_{j=\mu_1}^{\mu}\epsilon^jM_{\gamma_1,j}(t)
+\epsilon^\mu M_{\gamma_1,\mu_1}'M_{\gamma_2,\mu_2}
+o(\epsilon^{\mu}).\nonumber
\end{align}
Similarly, 
\begin{equation*}
P_{\gamma_2}\circ 
P_{\gamma_1}(t,\epsilon)=t
+\sum_{j=\mu_1}^{\mu}\epsilon^jM_{\gamma_1,j}(t)
+\sum_{j=\mu_2}^{\mu}\epsilon^jM_{\gamma_2,j}(t)
+\epsilon^\mu M_{\gamma_2,\mu_2}'M_{\gamma_1,\mu_1}
+o(\epsilon^{\mu}),\nonumber
\end{equation*}
and therefore
\begin{equation*}
P_{\gamma_2}\circ 
P_{\gamma_1}(t,\epsilon)=P_{\gamma_1}\circ 
P_{\gamma_2}(t,\epsilon)
+\epsilon^\mu\left(M_{\gamma_1,\mu_1}M_{\gamma_2,\mu_2}'
-M_{\gamma_1,\mu_1}'M_{\gamma_2,\mu_2}\right)
+o(\epsilon^{\mu}).
\end{equation*}
As $\left(P_{\gamma_2}^{-1}\circ P_{\gamma_1}^{-1}\right)'	=1+O(\epsilon)$,	
application of $ (P_{\gamma_2}^{-1}\circ 
P_{\gamma_1}^{-1}(s,\epsilon),\epsilon)$ provides the required equality
\begin{equation*}
P_\gamma=P_{\gamma_2}^{-1}\circ P_{\gamma_1}^{-1}\circ P_{\gamma_2}\circ 
P_{\gamma_1}
(t,\epsilon)=t
+\epsilon^\mu\left(M_{\gamma_1,\mu_1}M_{\gamma_2,\mu_2}'
-M_{\gamma_1,\mu_1}'M_{\gamma_2,\mu_2}\right)
+o(\epsilon^{\mu}).\qedhere
\end{equation*}
\end{proof}

\subsection{Explicit computations}

By Poincar\'e-Pontryagin criterion,  $M_{\sigma,1}=\int_\sigma\omega$. From 
\eqref{eq:delta eta pairing}
%
we have
\begin{align*}
(2\pi i)^{-1}\int_{\delta_1+\delta_2}\omega&=a_2(t),\\
(2\pi i)^{-1}\int_{\delta_2(t)}\omega&=a_2(t)-a_3(t), \\
(2\pi i)^{-1}\int_{\delta_2+\delta_3}\omega&=a_1(t)-a_3(t).
\end{align*}
By Proposition~\ref{prop:BCHformula} we have
	\begin{equation}
	\begin{aligned}
	(2\pi i)^{-2}M_{v_{2},2}(t)&=W(a_2(t),a_1(t)-a_3(t)),\\
	(2\pi i)^{-3}M_{v_3,3}(t)&=W\big(a_2(t),W\left(a_2(t)-a_3(t),
	a_1(t)-a_3(t)\right)\big).
	\end{aligned} \label{eq:Mv22}
	\end{equation}
	
	In what follows, the $2\pi i$ factors are 
	not important, so we will omit them.
	
	%
	%
	%
	
	%


\begin{lemma}\label{lem:ConditionsM2(v_2)=0,M_3(v_3)not0}
$M_{v_2,2}\equiv 0$ and $M_{v_3,3}\not\equiv 0$ if, and only if,
\begin{equation}
\begin{aligned}
a_3(t)&=\alpha_1\int_{0}^{t}\frac{\alpha_2(\tau)}{\alpha_1^2(\tau)}d\tau+c_0\alpha_1\\
a_1(t)&=a_3(t)+\alpha_1(t)\\
a_2(t)&=\lambda\alpha_1(t), \emph{ with }\lambda\in\C^*,
\end{aligned}
\end{equation}
where $\alpha_1(t)$ and $\alpha_2(t)$ are linearly independent functions over 
$\C$, and $\alpha_1$ is not constant.
\end{lemma}

\begin{proof}
From $M_{v_3,3}\not\equiv 0$ we see that $a_2(t), a_1(t)-a_3(t)\not\equiv0$.
 
 Then $M_{v_2,2}\equiv 0$ is equivalent to 
 \begin{equation}\label{a_2}
a_2(t)=\lambda_1 (a_1(t)-a_3(t)), \text{ for some }\lambda_1\in \C^*. 
\end{equation}



This implies, by linearity of Wronskians, 
\begin{equation}\label{eq:M3}
M_{v_3,3}= W\left(a_2(t),W(a_3(t),a_1(t))\right)=\lambda_1 
W\left(a_1(t)-a_3(t),W(a_3(t),a_1(t))\right).
\end{equation}
 
Then,  $M_{v_3,3}\not\equiv 0$ if, and only if,
\begin{equation}\label{CondM_3(v_3)neq0}
W(a_1,a_3)\neq \lambda_2(a_1(t)-a_3(t)),\text{ for all }\lambda_2\in\C.
\end{equation}
In other words,
\begin{equation}
\begin{aligned}
a_1(t)-a_3(t)&=\alpha_1(t)\\
a_3'(t)a_1(t)-a_1'(t)a_3(t)&=\alpha_2(t),
\end{aligned}\label{System-alpha_1alpha_2}
\end{equation}
where $\alpha_1(t)$ and $\alpha_2(t)$ are linearly independent functions over $\C$, and $\alpha_1$ is not constant, in order to get condition (\ref{CondM_3(v_3)neq0}).
\\

The solution of the system (\ref{System-alpha_1alpha_2}) is 
$$\begin{array}{rl}
a_1(t)&=a_3(t)+\alpha_1(t)\\
a_3(t)&=\alpha_1\int_{0}^{t}\frac{\alpha_2(\tau)}{\alpha_1^2(\tau)}d\tau
+c_0\alpha_1,\quad c_0\in\C.
\end{array}
$$ Substituting $a_1$ and $a_3$ in expression (\ref{a_2}) we get  $a_2(t)$.
\end{proof}

In general,    $M_{v_2,2}\equiv 0$ does not necessarily imply 
$M_{\gamma,2}(t)\equiv 0$. However,
\begin{lemma}\label{lem:M22 equivalent Mgamma2} 
	For a form $\omega$ of form \eqref{omega}, the condition 
$M_{v_2,2}(t)\equiv 0$ is equivalent to $M_{\gamma,2}(t)\equiv 0$.
\end{lemma}
\begin{proof}
Evidently, if $M_{\gamma,2}(t)\equiv 0$ then 
$M_{v_2,2}(t)=\Var(M_{\gamma,2}(t))\equiv 0$, so one implication is trivial.

Since $M_{\gamma,1}\equiv 0$  by Fran\c coise algorithm we have that 
$M_{\gamma,2}=\int_{\gamma}\omega'\omega$. 
Using integration by parts, we can rewrite $\omega$ as 
\begin{equation*}
\omega=\sum_{i=1}^{3}a_i(t)d\phi_i=
-\sum_{i=1}^{3}\phi_ia_i'(F)dF+d(\sum_{i=1}^{3}a_i(F)\phi_i),
\quad \phi_i=\log f_i.
\end{equation*} 
Denote $g=-\sum_{i=1}^{3}\phi_ia_i'(F)$ and $R=\sum_{i=1}^{3}a_i(F)\phi_i$. 
Then, $\omega'=dg$ and so $M_{\gamma,2}=\int_{\gamma(t)}g\omega$. Developing 
this expression we get

 \begin{equation}\label{M_2Francoise}
 M_{\gamma,2}=\sum_{1\le i<j\le 3}W(a_i(t),a_j(t))\int_{\gamma(t)}\phi_id\phi_j.
 \end{equation}

Next Lemma is  useful in following computations.
\begin{lemma}\label{lem:Cauchy}
	Assume that functions $f_i(x)$, $i=1,\dots,m$, are holomorphic in some 
	simply connected domain $U\subset\C_x$ containing the 
	projection $\gamma_x$ of $\gamma$ to the $x$-axis $\C_x$. Then the iterated 
	integral $\int_\gamma (f_1dx)\dots(f_mdx)$ vanishes.
\end{lemma}
\begin{proof} This integral is equal to $\int_{\gamma_x} (f_1dx)\dots(f_mdx)$, 
so this follows from Cauchy theorem.\end{proof}

By Lemma~\ref{lem:Cauchy}, 
$\int_{\gamma(t)}\phi_1d\phi_3=\int_{\gamma(t)}\log(x+1)\frac{dx}{x-1}\equiv 
0$. On the other hand, since 
$M_{v_2,2}=W\left(a_2(t),a_1(t)-a_3(t)\right)\equiv 0$, we have 
$W(a_1,a_2)+W(a_2,a_3)=0$, therefore
$$
M_{\gamma,2}=W(a_1(t),a_2(t))\left(\int_{\gamma(t)}\phi_1d\phi_2-\int_{\gamma(t)}\phi_2d\phi_3\right).
$$
Substituting $\phi_2d\phi_3$ by $-\phi_3d\phi_2+d(\phi_2\phi_3)$, and since $\int_{\gamma(t)}d(\phi_2\phi_3)=0$, we have
$$
M_{\gamma,2}=W(a_1(t),a_2(t))\left(\int_{\gamma(t)}\phi_1d\phi_2+\phi_3d\phi_2\right).
$$
Recall that $\phi_1=\log(x+1)$, $\phi_2=\log(y-1)$ and $\phi_3=\log(x-1)$, so
\begin{align*}
M_2(\gamma(t))&=W(a_1(t),a_2(t))\int_{\gamma(t)}(\log(x+1)+\log(x-1))\frac{dy}{y-1}\\
&=W(a_1(t),a_2(t))\int_{\gamma(t)}\log(x^2-1)\frac{dy}{y-1}\\
&=W(a_1(t),a_2(t)) 
\int_{\gamma(t)}\log\left(\frac{t}{y^2-1}\right)\frac{dy}{y-1}.
\end{align*}

Again, by Lemma~\ref{lem:Cauchy} this integral vanishes.
\end{proof}

\begin{proof}[Proof of Proposition~\ref{prop:pert3}]
Proposition~\ref{prop:pert3} follows from 
	Lemmas~\ref{lem:ConditionsM2(v_2)=0,M_3(v_3)not0},~\ref{lem:M22 
	equivalent Mgamma2}.
\end{proof}

\subsubsection{$M_{\gamma,2}=M_{\gamma,3}=0$ implies 
center.}\label{M2=M3=0}

\begin{proposition}\label{prop:M3=0}
	For deformation \eqref{def} with $\omega$ as in \eqref{omega}, identical 
	vanishing of both 
	$M_{\gamma,2},M_{\gamma,3}$ is equivalent to preservation of the center.
\end{proposition}

One implication is trivial, so we assume $M_{\gamma,2}=M_{\gamma,3}\equiv0$ and 
prove that \eqref{def} preserves the center.

First, there is a trivial symmetric case.
\begin{lemma}
	If either $a_2(t)\equiv0$ or $a_1(t)-a_3(t)\equiv0$, then \eqref{def} 
	defines a 
	center.
\end{lemma}
\begin{proof}
Indeed, then the foliation \eqref{def} is symmetric with respect 
to the 
symmetry $y\to -y$ 
or with respect to the symmetry $x\to -x$, correspondingly.\end{proof}

Further, we assume that $a_1(t)-a_3(t), a_2(t)\not\equiv0$. 
The conditions $M_{\gamma,2}=M_{\gamma,3}=0$ imply $M_{v_2,2}=M_{v_3,3}=0$. 
From 
\eqref{eq:Mv22}\eqref{eq:M3}, 
this is equivalent to
\begin{equation}
\begin{aligned}
a_1-a_3&=\lambda_1 a_2, \quad \lambda_1\in\C^*\\
W(a_1,a_3)&=\lambda_2 a_2, \quad \lambda_2\in\C.
\end{aligned}
\end{equation}
This implies 
\begin{equation}
W(a_1,a_2)=-a_2^2\left(\frac{a_1}{a_2}\right)'=-\lambda a_2, \quad 
\lambda=\frac{\lambda_2}{\lambda_1},
\end{equation}
i.e. necessarily
\begin{equation}
\begin{aligned}
a_1&=a_2\left(\lambda\int\frac{dt}{a_2}+c_1\right)\\
a_3&=a_2\left(\lambda\int\frac{dt}{a_2}+c_1-\lambda_1\right).
\end{aligned}
\end{equation}

Denote $A(t)=\int\frac{dt}{a_2}$, so 
\begin{equation}\label{g(t)omega}
\omega=a_2(t)\Big[
A(t)\alpha
+d\phi\Big], \qquad \alpha=\lambda\big(d\phi_1+d\phi_3\big),\quad\phi=c_1 
\phi_1 + \phi_2+(c_1-\lambda_1) 
\phi_3.
\end{equation}

As  $a_2(t)=A'(t)^{-1}$, the foliation  \eqref{def} is then  equal to 
	\begin{equation}\label{dF+epsilonomega=0}
	dF+\epsilon\frac{1}{A'(F)}( A(F)\alpha+d\phi)=0,
	\end{equation}
	which is orbitally equivalent to 
	\begin{equation}\label{eq:da+epsilonomega}
dA(F)+\epsilon\eta=0, \qquad\eta= A(F)\alpha+d\phi.
	\end{equation}
	If $\lambda=0$, then $\alpha=0$, and this is a Hamiltonian system with 
	Hamiltonian $A(F)+\epsilon\phi$ (recall that $\phi$ is a holomorphic 
	function in a 
	neighborhood of  $\gamma$). Thus $\lambda=0$ implies preservation of 
	center, and it remains to prove that vanishing of $M_{\gamma,3}$ implies 
	$\lambda=0$.
	
	We consider \eqref{eq:da+epsilonomega} as a perturbation of a Hamiltonian 
	system with Hamiltonian $A(F)$.

\begin{lemma}
	The first non-zero Melnikov functions $M_{\gamma,k}$ of 
	\eqref{dF+epsilonomega=0} and the first non-zero Melnikov functions 
	$\tilde{M}_{\gamma,k}$ of \eqref{eq:da+epsilonomega} are related by 
		\begin{equation}\label{tilde(M)=gM}
	\tilde{M}_k=g'(F)M_k.
	\end{equation}
\end{lemma}
\begin{proof}
	Indeed, the passage from \eqref{dF+epsilonomega=0} to  
\eqref{eq:da+epsilonomega} amounts to reparameterization of the transversal 
$\tau$ by values of $A(F)$. For a point $p\in\tau$ 
with $F(p)=t$,
\begin{equation*}
\tilde{M}_{\gamma,k}=\frac{\partial 
A(F(P(\epsilon,p)))}{\partial\epsilon^k}=A'(t)\frac{\partial 
F(P(\epsilon,p))}{\partial\epsilon^k}=A'(t)M_{\gamma,k}.\qedhere
\end{equation*}\end{proof}
\begin{remark}
	In other words, the first non-zero Melnikov function has tensor type of a 
	vector field, which is expected from Proposition~\ref{prop:BCHformula} and 
	the 
	following Remark.
\end{remark}

\begin{proof}[Proof of Proposition~\ref{prop:M3=0}]

We will compute $\tilde{M}_{\gamma,3}$.
	By Fran\c coise's algorithm \cite{F,G}, we have that 
	\begin{equation}
	\tilde{M}_{\gamma,3}=\int_{\gamma}(\eta'\eta)'\eta,
	\end{equation}
	where the Gelfand-Leray derivative is taken with respect to the Hamiltonian 
	$A(F)$.
	Developing the derivative, 
	\begin{equation}
	\tilde{M}_{\gamma,3}=\int_{\gamma}\eta''\eta\eta+\int_{\gamma}\eta'\eta'\eta
	+\int_{\gamma}\frac{\eta'\wedge\eta}{dA}\eta,
	\end{equation}
	where $\eta'=\frac{d(A\alpha+d\phi)}{dA}=\alpha$, and $\eta''=0$, as the 
	forms $\alpha,d\phi$ are closed. Then, 
\begin{equation}
\tilde{M}_{\gamma,3}=\int_{\gamma}\alpha\alpha(A\alpha+d\phi)+
\int_{\gamma}\frac{\alpha\wedge(A\alpha+d\phi)}{dA}\eta.	
\end{equation}
By Lemma~\ref{lem:Cauchy} the triple integrals 
$\int_\gamma\alpha\alpha d\phi_1,\int_\gamma\alpha\alpha d\phi_3$ 
vanish. Also, as 
    \begin{equation}\label{eq:alpha in y}
    \alpha=\lambda d\log F -\lambda d\log(y^2-1),
    \end{equation} 
Lemma~\ref{lem:Cauchy} implies that the integral 
$\int_\gamma\alpha\alpha d\phi_2$ vanishes. 

Hence, 
	\begin{equation}
	\tilde{M}_{\gamma,3}=\int_{\gamma}\frac{\alpha\wedge(A\alpha+d\phi)}{dA}\eta=
	\int_{\gamma}\frac{\alpha\wedge d\phi}{dA}\eta.
	\end{equation}
	
We have 
	$\alpha\wedge d\phi=\alpha\wedge d\phi_2$. By 
	\eqref{eq:alpha in y}, we have
	$\alpha\wedge d\phi_2=\frac{\lambda dF}{F}\wedge d\phi_2$, so 
\begin{equation}
	\frac{\alpha\wedge d\phi}{dA}=\frac{\lambda\frac{dF}{F}\wedge 
	d\phi_2}{A'(F)dF}=\frac{\lambda}{FA'(F)}d\phi_2,
\end{equation}
and, as $\eta=A(F)\alpha+d\phi$, 
     \begin{equation} 	
     \tilde{M}_{\gamma,3}=\frac{\lambda}{tA'(t)}\int_{\gamma}d\phi_2\eta
     =\frac{\lambda}{A'(t)}\int_{\gamma}d\phi_2\alpha+\frac{\lambda}{tA'(t)}\int_{\gamma}d\phi_2
      d\phi.
     \end{equation} Again, 
	$\int_{\gamma}d\phi_2\alpha\equiv 0$ by \eqref{eq:alpha in y} and 
	Lemma~\ref{lem:Cauchy}. Moreover,
	$$
	\int_{\gamma}d\phi_2 d\phi=\int_{\gamma}d\phi_2d\phi_2
	+\frac{c_1}{\lambda}\int_{\gamma}d\phi_2\alpha
	-\lambda\int_{\gamma}d\phi_2d\phi_3=-\lambda\int_{\gamma}d\phi_2d\phi_3,
	$$
	 as $\int_{\gamma}d\phi_2d\phi_2\equiv 0$ by Lemma~\ref{lem:Cauchy}. 
	 Therefore
\begin{equation}
\tilde{M}_{\gamma,3}=\frac{-\lambda^2}{tA'(t)}\int_{\gamma}d\phi_2d\phi_3. 
\end{equation}
We claim that
\begin{equation}\label{eq:INTdf2df3not0}
	\int_{\gamma}d\phi_2d\phi_3\not\equiv 0.
	\end{equation}
Indeed, 
	from  $\int_\gamma d\phi_i\equiv0$ and 
	$\int_{[\sigma_1,\sigma_2]}\omega_1\omega_2
	=\det\left\{\int_{\sigma_i}\omega_j\right\}_{i,j=1}^2$, we see that $\int 
	d\phi_2 d\phi_3$ vanishes on $K$, and, therefore, defines a linear 
	functional on $(O/K)^*$. Therefore, 
	\begin{equation*}
	Var^2\int_{\gamma}d\phi_2d\phi_3=\int_{v_2}d\phi_2d\phi_3
	=\det
	\begin{pmatrix}
	\begin{array}{cc}
\int_{\delta_1+\delta_2}d\phi_2	&\int_{\delta_1+\delta_2}d\phi_3  \\ 
\int_{\delta_2+\delta_3}d\phi_2	&\int_{\delta_2+\delta_3}d\phi_3
	\end{array} 
	\end{pmatrix}=4\pi^2\end{equation*}
 by \eqref{eq:delta eta pairing}, which proves \eqref{eq:INTdf2df3not0}.

Thus $\tilde{M}_{\gamma,3}$, as well as  ${M}_{\gamma,3}$, vanish identically 
only if $\lambda=0$, which finishes the proof of 
Proposition~\ref{prop:M3=0}.\end{proof}
\begin{proof}[Proof of Theorem~\ref{thm:M3}]
	Take $\omega$ as in \eqref{eq:omega type}, with coefficients $a_i(F)$ as in 
	Proposition~\ref{prop:pert3}. The first Melnikov function $M_{\gamma,1}$ 
	vanishes identically for all $\omega$ of this type. Also, $M_{\gamma,2}$ 
	vanishes by Proposition~\ref{prop:pert3}. Therefore $M_{\gamma,3}$ is the 
	first non-zero Melnikov function of \eqref{def} or it is identically zero. 
	In 	both cases, it is linear on the orbit, see \cite{GI,MNOP}, and 
	therefore 	$Var^3\left(M_{\gamma,3}(t)\right)=M_{v_3,3}(t)\not=0$. 
	Therefore 	$M_{\gamma,3}(t)\not=0$, and, moreover, has length three.
	Taking $\alpha_1=t$, $\alpha_2=t^2$ and $c_0=\lambda=1$, one gets the 
	example of Theorem~\ref{thm:M3}.

	The last statement of Theorem~\ref{thm:M3} follows from 
	Proposition~\ref{prop:M3=0}.\end{proof}

\subsection{Length bigger than 4.}\label{Lengthbiggerthan4}

Now,  for deformation \eqref{def} consider the functions $M_{v_i,i}$, where 
$v_i$ were defined in 
Proposition~\ref{prop:variations}. Note that  $M_{v_i,j}$, given by iterated 
integrals of length at most $j$, necessarily vanish on $v_i\in L_i$ for $j<i$, 
so $M_{v_i,i}$ are (generically) the first non-zero Melnikov functions of 
$v_i$  with respect to the 
deformation \eqref{def}.
\begin{lemma}
The condition $M_{v_2,2}=M_{v_3,3}=0$ implies 
	$M_{v_i,i}=0$ for all $i\ge 4$.
\end{lemma}
\begin{remark}
	Vanishing of $M_{v_i,i}$ is necessary for vanishing of $M_{\gamma,i}$ 
	(i.e.  follows from center conditions), but not sufficient, see for example 
	Proposition~\ref{prop:M3=0}.
\end{remark}
\begin{proof}
	
		Denote $\beta_1=\int_{\delta_1+\delta_2}\omega$, 
	$\beta_2=\int_{\delta_2}\omega$ and 
	$\beta_3=\int_{\delta_2+\delta_3}\omega$. 
	By Proposition~\ref{prop:BCHformula}, 	
\begin{equation*}
	\begin{aligned}
	M_{v_2,2}&=W(\beta_1,\beta_3)\\
	M_{v_3,3}&=W(\beta_1,W(\beta_2,\beta_3))\\
	M_{v_4,4}&=W(\beta_1,W(\beta_2,W(\beta_2,\beta_3)))\\
	\vdots\\
	M_{v_i,i}&=W(\beta_1,W(\beta_2,\dots,W(\beta_2,\beta_3))\dots).
	\end{aligned}.
\end{equation*}

	Suppose $M_{v_2,2}\equiv 0$. If $\beta_3\equiv0$ then evidently all these 
	Wronskians vanish. Otherwise, $\beta_1=\lambda_1\beta_3$, for some 
	$\lambda\in 
	\C$.
	
	Suppose also that $M_3(v_3)\equiv 0$. Again, the case $\beta_1\equiv0$ is 
	trivial. Otherwise,  $W(\beta_2,\beta_3)=\lambda_2\beta_1$, for some 
	$\lambda_2\in \C$, and therefore
	$$W(\beta_2,\beta_3)=\lambda_1\lambda_2\beta_3,$$ 
	which implies 
	$$
	M_{v_i+1,i+1}=\lambda_1\lambda_2M_{v_i,i}=\dots
	=(\lambda_1\lambda_2)^{i-2}M_{v_3,3}\equiv0.\qedhere
	$$
\end{proof}

\end{document}